\documentclass{article}
\usepackage[margin=1 in]{geometry}
\usepackage{lineno,rotating}
\usepackage{amsmath,amssymb,amsfonts,amsthm,anysize,setspace}
\usepackage{graphicx}
\usepackage{hyperref}
\usepackage{graphicx}
\usepackage{subfigure}
\usepackage{pdftexcmds}
\usepackage{authblk}
\usepackage{mathtools}

\newtheorem{example}{Example}

\newtheorem{remark}{Remark}
\newtheorem{definition}{Definition}
\newtheorem{theorem}{Theorem}
\newtheorem{lemma}{Lemma}

\begin{document}


\title{High order approximation to Caputo derivative on graded mesh and time-fractional diffusion equation for non-smooth solutions}


\author[1]{Shweta Kumari}
\affil{Deptartment of Mathematics\\ Indian Institute of  Technology Delhi, India\\
\texttt{shwetakri8062@gmail.com, assinghabhi@gmail.com, vaibhavmehandiratta@gmail.com, mmehra@maths.iitd.ac.in}}
 
\author[1]{Abhishek Kumar Singh}
\author[1]{Vaibhav Mehandiratta}
\author[1]{Mani Mehra}

\maketitle

\begin{abstract}
In this paper, a high-order approximation to Caputo-type time-fractional diffusion equations involving an initial-time singularity of the solution is proposed. At first, we employ a numerical algorithm based on the Lagrange polynomial interpolation to approximate the Caputo derivative on the non-uniform mesh. Then truncation error rate and the optimal grading constant of the approximation on a graded mesh are obtained as $\min\{4-\alpha,r\alpha\}$ and $\frac{4-\alpha}{\alpha}$, respectively, where $\alpha\in(0,1)$ is the order of fractional derivative and $r\geq 1$ is the mesh grading parameter. Using this new approximation, a difference scheme for the Caputo-type time-fractional diffusion equation on graded temporal mesh is formulated. The scheme proves to be uniquely solvable for general $r$. Then we derive the unconditional stability of the scheme on uniform mesh. The convergence of the scheme, in particular for $r=1$, is analyzed for non-smooth solutions and concluded for smooth solutions. Finally, the accuracy of the scheme is verified by analyzing the error through a few numerical examples.
\end{abstract}


\numberwithin{equation}{section}
\section{\textbf{Introduction}}\noindent
The classical integer-order differential equations are incapable of framing several daily life occurrences due to their complex structure and dependency of the physical system on previous impressions. However, the fractional operators are known mainly for their nonlocal behavior and past memory storage. Hence, the involvement of differential equations with fractional order, namely the fractional differential equations (FDEs), is crucial in molding such phenomena. They can be particularised by generalizing the order of the integer-order derivative to some arbitrary value. Although both calculus and fractional calculus trace their roots back to the same time period, i.e. the late 17th century, the latter did not receive ample attention as no practical usage of it was apparent at the time. But in the past few years, the study of FDEs upsurged due to their wide applications in the fields of science and engineering \cite{kilbas2006theory,podlubny1998fractional}. FDEs are generally used to describe the dynamic systems that represent many natural processes in physics~\cite{hilfer2000fractional}, fluid mechanics~\cite{mainardi1997model}, biology~\cite{iomin2004tumor,singh2023modified}, system control~\cite{mehandiratta2021optimal,mehandiratta2021fractional}, wavelet~\cite{Nitin,singh2021wavelet}, network like structures~\cite{MEHANDIRATTA2020152}, finance~\cite{patel2020fourth,singh2020uncertainty}  and various other engineering fields~\cite{kumar2021collocation,singh2023learning,shukla2020fast}. Among fractional partial differential equations (FPDEs), fractional diffusion equations are attracting strong interests because of their ability to well specify anomalous diffusion processes~\cite{luchko2012anomalous} in complex media since classical diffusion is not the best description for them.\par
This paper is concerned with developing a numerical scheme for the Caputo-type time-fractional diffusion equation (TFDE) for the following initial-boundary value problem (IBVP):
\begin{equation}\label{diff_eqn}
\textsubscript{C}D_{0,t}^{\alpha}u(x,t)- \varrho\partial_x^{2}u(x,t)=f(x,t),~~ 0<t\leq T,~0<x<L,\\
\end{equation}
\begin{equation}\label{bdry_eqn}
u(0,t) = u(L,t) = 0,~~~0 < t \leq T,\\
\end{equation}
\begin{equation}\label{ini_eqn}
u(x,0) = \varphi(x),~~~0\leq x\leq L,
\end{equation}
where $\textsubscript{C}D_{0,t}^{\alpha}$ denotes the left Caputo fractional derivative of order $\alpha \in (0,1)$, with respect to time variable $t$, $\partial_x^2$ is the spatial Laplacian operator, \eqref{bdry_eqn} denotes the Dirichlet-type boundary conditions, and \eqref{ini_eqn} denotes the initial condition. The notation $\varrho$ $(>0)$ is a constant diffusion coefficient. Here, $f \in C(\bar \Omega)$ and $ \varphi \in C[0,L]$, where $\bar \Omega=[0,L]\times [0,T]$. \par
Recently, intense focus has been placed on the approximation to Caputo derivative as it models phenomena with productive memory functions that take account of past interactions and problems with nonlocal properties \cite{LI20113352}. The classic $L1$ format to approximate the Caputo derivative on a uniform grid is derived by Lin and Xu in \cite{LIN20071533}, where the integrand of the Caputo derivative is approximated by a linear interpolation polynomial in each time subinterval. For the subdiffusion equation with non-smooth data, the $L1$ scheme was analysed recently in \cite{jin2016analysis}. Gao et al.~\cite{GAO201433} proposed the $L1-2$ format for Caputo fractional derivative on a uniform grid, i.e. a linear and quadratic interpolation polynomial is used to discretize the integrand in first and other time subintervals, respectively, with a convergence order of $(3-\alpha)$. Other works for higher order numerical algorithms having $(3-\alpha)$ convergence rate are presented in \cite{li2014high,alikhanov2015new}. In \cite{li2014high}, an approximation of order $(3-\alpha)$ was derived for the Caputo-type advection-diffusion equation. Alikhanov \cite{alikhanov2015new} gave the $L2-1_\sigma$ format to approximate the fractional derivatives on uniform grid, where $\sigma=1-\frac{\alpha}{2}$ is the offset parameter. A scheme of order $(4-\alpha)$ is developed by Cao et al.~\cite{Cao_2015} for time fractional advection-diffusion equations having smooth solutions throughout the time interval on a uniform mesh. \par
Many researchers have obtained the numerical solution of the TFDE \eqref{diff_eqn}-\eqref{ini_eqn} with the assumption that the solution and its consecutive derivatives satisfy certain smooth conditions. Nevertheless, Stynes et al.~\cite{Stynes_SIAM} were the first to observe that the solution of the problem \eqref{diff_eqn}-\eqref{ini_eqn} usually has an initial-time singular behavior. In that case, however, the numerical schemes proposed over the uniform grid converge, but at the cost of a significantly reduced error convergence rate. So, they presented their work~\cite{Stynes_SIAM} for the time-fractional reaction-diffusion problem to deal with this initial-time singularity on a non-uniform grid with the ideal convergence order of $\min\{2-\alpha,r\alpha\}$, where $\alpha \in (0,1)$ is the fractional order and $r\geq 1$ is the mesh grading parameter. They used the $L1$ format to discretize the Caputo fractional derivative on graded temporal meshes, and the diffusion term is discretized by the central difference approximation on a uniform spatial grid. Following a similar path, Qiao and Cheng~\cite{qiao2022fast} successfully got a truncation error of order $\min\{3-\alpha,r\alpha\}$ using the $L1-2$ type format to discretize the Caputo derivative on the graded mesh. But the examination of the stability and convergence of the scheme was missing.\par
Taking the aforementioned works as motivation, we attempt to develop an even higher order approximation of the Caputo derivative on graded mesh and use it to construct a difference scheme to TFDE \eqref{diff_eqn} having non-smooth solutions. To achieve this, we apply a numerical algorithm to discretize the Caputo derivative on graded temporal mesh in which the linear, quadratic, and cubic interpolation polynomials are used to approximate the integrand of the Caputo derivative in the first, second, and third onwards time subintervals, respectively. The diffusion term is discretized by second order central difference scheme on a uniform spatial grid. This algorithm requires the solution to be at least fourth-order continuously differentiable in time. However, solutions usually do not satisfy such high regularity assumptions, which results in a low error convergence rate of the scheme. Hence, the grading of mesh is considered so that the initial-time singularity does not affect the overall convergence of the scheme.
Surprisingly, there was a lack of literature on the convergence analysis of the scheme for TFDEs with smooth solutions as well. So, we also investigate the convergence order for TFDEs with smooth solutions on a uniform mesh.\par
The rest organization of this paper goes as follows: Section~\ref{preliminaries} contains the preliminary content required for the reader to understand the concepts of this paper. The higher order approximation of Caputo derivative on graded mesh has been constructed and thoroughly analyzed in Section~\ref{derivation}. It also comprises the estimation of the local truncation error of the approximation. Section~\ref{scheme} builds a numerical scheme for the TFDE problem \eqref{diff_eqn}-\eqref{ini_eqn} using the approximation developed in Section~\ref{derivation}. Further, we check the uniqueness of the solution of the scheme. We dedicate a subsection to study the unconditional stability and convergence of the scheme developed for $r=1$, i.e. on uniform mesh and for both non-smooth as well as smooth solutions. In Section~\ref{results}, a few numerical examples are displayed to verify the accuracy and efficiency of the scheme by studying the errors. Finally, Section~\ref{conclusion} concludes the paper.

\section{Preliminaries}
\label{preliminaries}\noindent 
In this part, we present some fundamental concepts that are prerequisites for a better understanding of our work in this paper. 

\begin{definition}\label{gen_caputo}
The left Caputo derivative with fractional order $\alpha > 0$ of the function $f(t),\\ f \in C^m(a,b)$, is defined as \cite{li2015numerical}
\begin{equation}
\textsubscript{C}D_{a,t}^{\alpha}f(t) = \frac{1}{\Gamma(m-\alpha)}\int_{a}^{t}(t-s)^{m-\alpha-1}f^{(m)}(s)ds,
\end{equation}
where $m \in  \mathbb{N}$ satisfying $m-1 < \alpha \leq  m$ \ and $\Gamma(.)$ is the Euler's gamma function .
\end{definition}
\begin{definition}
Euler's gamma function is denoted and defined as
\begin{equation*}
    \Gamma(z)=\int_0^\infty x^{z-1} e^{-x}dx,~~
  z \in \mathbb{C},~\Re(z)>0 .
\end{equation*}
Here $\Gamma(z+1)=z\Gamma(z)$, and in particular,
 $\Gamma(z)=(z-1)!$ \ for $z \in \mathbb{Z}^+$.
\end{definition}
\begin{definition}
The linear interpolation of a function $f(t)$ on some interval $[t_{k-1},t_k],~k\geq 1,$ is defined as
\begin{equation}\label{lin_int}
I_{1k}f(t) = \frac{t-t_k}{t_{k-1}-t_k}f(t_{k-1})+\frac{t-t_{k-1}}{t_k-t_{k-1}}f(t_k).
\end{equation}
The quadratic interpolation of a function $f(t)$ on some interval $[t_{k-1},t_k],~k\geq 2,$ is defined as
\begin{equation}\label{quad_int}
\begin{aligned}
I_{2k}f(t) &= \frac{(t-t_{k-1})(t-t_k)}{(t_{k-2}-t_{k-1})(t_{k-2}-t_k)}f(t_{k-2})+\frac{(t-t_{k-2})(t-t_k)}{(t_{k-1}-t_{k-2})(t_{k-1}-t_k)}f(t_{k-1})\\&\hspace{.5cm}+\frac{(t-t_{k-2})(t-t_{k-1})}{(t_{k}-t_{k-2})(t_{k}-t_{k-1})}f(t_k),
\end{aligned}
\end{equation}
and the cubic interpolation of a function $f(t)$ on some interval $[t_{k-1},t_k],~k\geq 3,$ is constructed in the following form
\begin{equation}\label{cub_int}
I_{3k}f(t) = \sum_{l=0}^{3}{f(t_{k-l})}\prod_{i=0,i\neq l}^{3}\frac{t-t_{k-i}}{t_{k-l}-t_{k-i}}.
\end{equation}
Thus, we have the interpolation error from \cite{qiao2022fast} as
\begin{equation}\label{int_error}
f(t)-I_{3k}f(t)=\frac{f^{(iv)}(\xi_k)}{12}(t-t_{k-3})(t-t_{k-2})(t-t_{k-1})(t-t_{k}),
\end{equation}
where $\xi_k \in (t_{k-3},t_k),~k\geq3$.
\end{definition}
\begin{definition}
The Supremum norm of a function $f(x)$, $x \in X$, is defined as
\begin{equation*}
    \|f\|_\infty=\sup_{x\in X}|f(x)|.
\end{equation*}
Then for any mesh function ${u_i^n}=u(x_i,t_n)$, $(x_i,t_n)\in \bar\Omega_{h}\times\bar\Omega_{\tau}$,
\begin{equation*}
\|u^n\|_\infty = \max_{0\leq i\leq K_x}|u_i^n| \ \text{ and } \
\|u\|_\infty = \max_{0\leq n\leq K_t}\|u^n\|_\infty.
\end{equation*}
\end{definition}
\section{Derivation and analysis of the higher order approximation to Caputo derivative on graded mesh}\label{derivation}
\noindent
To approximate the solution of TFDE \eqref{diff_eqn}-\eqref{ini_eqn}, the Caputo derivative is discretized on graded mesh, and the diffusion term is discretized on a uniform spatial grid. For graded temporal discretization, we discretize the time domain $[0,T]$ into $K_{t}$ subintervals with $t_{0}=0 < t_{1}< \cdots < t_{K_{t}} = T$. From here onwards, we denote $\mathbb{Z}_{[0,M]}:=\{z \in \mathbb{Z}$, $0\leq z\leq M\}$. We define $\tau_n={t_n}-t_{n-1}$ as the grid size in time direction with $t_{n} = T{\left(\frac{n}{K_{t}}\right)}^r$, where $n \in \mathbb{Z}_{[0,K_t]}$ and $r\geq1$ is the mesh grading parameter. Meanwhile, the spatial interval $[0,L]$ is divided into $K_{x}$ equal subintervals of uniform length $h = \frac{L}{K_{x}}$, where $K_{x} \in \mathbb{Z^+}$. Thus, the temporal and spatial meshes are given by 
$$\bar\omega_{\tau} = \{t_{n} : t_{n} =  T{\Big(\frac{n}{K_{t}}\Big)}^r,n \in \mathbb{Z}_{[0,K_t]}\},~\text{and}~\bar\omega_{h} = \{x_{i} : x_{i} = ih, i \in \mathbb{Z}_{[0,K_{x}]}\},$$
respectively. Then the domain $\bar\Omega = [0,L]\times [0,T]$ is covered by $\bar\omega_{h\tau} := \bar\omega_{h}\times \bar\omega_{\tau}$. We define the following grid functions as
\begin{equation*}
u(x_{i},t_{n})=U_{i}^{n} \text{ and } f(x_{i},t_{n})=f_{i}^{n}.
\end{equation*}
 Moreover, we use the second-order central finite difference scheme to approximate the diffusion term in problem \eqref{diff_eqn} as
\begin{equation}\label{space_eqn}
\begin{aligned}
{\partial_{x}^{2}}u(x_i,t_n) &= \frac{U_{i+1}^{n}-2U_{i}^{n}+U_{i-1}^{n}}{h^2}+\mathcal{O}(h^2),\\
&= \delta_x^2U_i^n+\mathcal{O}(h^2).
\end{aligned}
\end{equation}
From \eqref{gen_caputo}, the Caputo time-fractional derivative of a function $u(x,t)$ at mesh point $(x_i,t_n)$ and order $\alpha\in(0,1)$ is denoted and defined as
\begin{align}\notag
\textsubscript{C}D_{0,t}^{\alpha}u(x_i,t_n) &= \frac{1}{\Gamma(1-\alpha)}\int_{0}^{t_{n}}(t_{n}-s)^{-\alpha}{\frac{\partial u(x_i,s)}{\partial {s}}}ds, \\\label{caputo}
&= \frac{1}{\Gamma(1-\alpha)}\sum_{k=1}^{n}\int_{t_{k-1}}^{t_{k}}(t_{n}-s)^{-\alpha}{\frac{\partial u(x_i,s)}{\partial {s}}}ds.
\end{align}
\subsection{\bf{Approximation to Caputo fractional derivative}}\noindent
For the foremost subinterval $[t_0,t_1]$, we apply linear interpolation \eqref{lin_int} to approximate \eqref{caputo}, thus we get
\begin{align}\notag
\frac{1}{\Gamma(1-\alpha)}{\int_{t_0}^{t_1}}({{t_n}-s})^{-\alpha}{\frac{\partial{u}(x_i,s)}{\partial {s}}}ds 
& \approx \frac{1}{\Gamma(1-\alpha)}{\int_{t_0}^{t_1}}({{t_n}-s})^{-\alpha}{\frac{\partial{I_{11}u}(x_i,s)}{\partial {s}}}ds,\\\notag
& = \frac{U_i^1-U_i^0}{\tau_1 \Gamma(1-\alpha)}{\int_{t_0}^{t_1}}{({t_n}-s)}^{-\alpha}ds,\\\label{e2}
& = \frac{a_n}{\Gamma(2-\alpha)}\big[U_i^1-U_i^0\big],
\end{align}
where \ $a_n = {\frac{1}{\tau_1}}\left[({t_n}-{t_0})^{1-\alpha}-({t_n}-{t_1})^{1-\alpha}\right].$
\par
\bigskip
Now for the second subinterval $[t_1,t_2]$, using quadratic interpolation $\eqref{quad_int}$, we approximate \eqref{caputo} as follows
\begin{equation}\label{e3}
\begin{aligned}
\frac{1}{\Gamma(1-\alpha)}{\int_{t_1}^{t_2}}({{t_n}-s})^{-\alpha}{\frac{\partial u(x_i,s)}{\partial {s}}}ds 
&\approx \frac{1}{\Gamma(1-\alpha)}{\int_{t_1}^{t_2}}({{t_n}-s})^{-\alpha}{\frac{\partial{I_{22}u}(x_i,s)}{\partial{s}}}ds,\\
&=\frac{1}{\Gamma(2-\alpha)}\big[b_n{U_i^0}+c_n{U_i^1}+d_n{U_i^2}\big],
\end{aligned}
\end{equation}
where 
\begin{equation*}
\begin{aligned}
b_n =& {\frac{1}{\tau_1({\tau_1}+{\tau_2})}}\bigg[{\frac{2}{2-\alpha}}\big\{({t_n}-{t_1})^{2-\alpha}-({t_n}-{t_2})^{2-\alpha}\big\}-\tau_2\big\{({t_n}-{t_1})^{1-\alpha}+({t_n}-{t_2})^{1-\alpha}\big\}\bigg],\\
c_n =& -{\frac{1}{{\tau_1}{\tau_2}}}\bigg[{\frac{2}{2-\alpha}}\big\{({t_n}-{t_1})^{2-\alpha}-({t_n}-{t_2})^{2-\alpha}\big\}+({\tau_1}-{\tau_2})({t_n}-{t_1})^{1-\alpha}-({\tau_1}+{\tau_2})({t_n}-{t_2})^{1-\alpha}\bigg],\\
d_n = &{\frac{1}{\tau_2({\tau_1}+{\tau_2})}}\bigg[{\frac{2}{2-\alpha}}\big\{({t_n}-{t_1})^{2-\alpha}-({t_n}-{t_2})^{2-\alpha}\big\}+\tau_1({t_n}-{t_1})^{1-\alpha}-(2\tau_2+\tau_1)({t_n}-{t_2})^{1-\alpha}\bigg].\\\vspace{.2cm}
\end{aligned}
\end{equation*}
Finally for other subintervals $[t_{j-1},t_j]$, $3\leq j\leq n$, using a cubic interpolation function \eqref{cub_int} to approximate \eqref{caputo} provides
\begin{align}\notag
{\frac{1}{\Gamma(1-\alpha)}}\sum_{j=3}^{n}{\int_{t_{j-1}}^{t_j}}({{t_n}-s})^{-\alpha}{\frac{\partial u(x_i,s)}{\partial{s}}}ds
& \approx {\frac{1}{\Gamma(1-\alpha)}}\sum_{j=3}^{n}{\int_{t_{j-1}}^{t_j}}({{t_n}-s})^{-\alpha}{\frac{\partial{I_{3j}u}(x_i,s)}{\partial {s}}}ds,\\\notag
& = \frac{1}{\Gamma(2-\alpha)}\sum_{k = 3}^{n}\Big[w_{1,k}^{n}U_i^k+w_{2,k}^{n}U_i^{k-1}+w_{3,k}^{n}U_i^{k-2}\\\label{e4}
&\hspace{2.6cm}+w_{4,k}^{n}U_i^{k-3}\Big],
\end{align}
where
\begin{equation*}
\begin{aligned}
w_{1,k}^{n} & =
A_1\Big[\tau_{k-1}(\tau_{k-1}+\tau_{k-2}){(t_n-t_{k-1})}^{(1-\alpha)}-\big\{(\tau_k+\tau_{k-1})(2\tau_k+\tau_{k-1}+\tau_{k-2})\\
&\hspace{.4cm}+\tau_k(\tau_k+\tau_{k-1}+\tau_{k-2})\big\}{(t_n-t_k)}^{1-\alpha}+\frac{2}{(2-\alpha)}\big\{{(2\tau_{k-1}+\tau_{k-2})}{(t_n-t_{k-1})^{2-\alpha}}\\
&\hspace{.4cm}-(3\tau_k+2\tau_{k-1}+\tau_{k-2})(t_n-t_k)^{2-\alpha}\big\}+\frac{6}{(2-\alpha)(3-\alpha)}\big\{{(t_n-t_{k-1})}^{3-\alpha}-{(t_n-t_k)}^{3-\alpha}\big\}\Big],
\end{aligned}
\end{equation*}
\begin{equation*}
\begin{aligned}
w_{2,k}^{n} & = 
A_2\Big[\big\{\tau_{k-1}(\tau_k-\tau_{k-1}-\tau_{k-2})+\tau_k(\tau_{k-1}+\tau_{k-2})\big\}{(t_n-t_{k-1})}^{(1-\alpha)}\\
&\hspace{.4cm}+(\tau_k+\tau_{k-1}+\tau_{k-2})(\tau_k+\tau_{k-1}){(t_n-t_k)}^{1-\alpha}+\frac{2}{(2-\alpha)}\big\{(\tau_k-2\tau_{k-1}-\tau_{k-2})(t_n-t_{k-1})^{2-\alpha}\\
&\hspace{.4cm}+{(2\tau_k+2\tau_{k-1}+\tau_{k-2})}{(t_n-t_{k})^{2-\alpha}}\big\}-\frac{6}{(2-\alpha)(3-\alpha)}\big\{{(t_n-t_{k-1})}^{3-\alpha}-{(t_n-t_k)}^{3-\alpha}\big\}\Big],\\
w_{3,k}^{n} & = 
A_3\Big[-\tau_k(\tau_{k-1}+\tau_{k-2}){(t_n-t_{k-1})}^{1-\alpha}-\tau_k(\tau_k+\tau_{k-1}+\tau_{k-2}){(t_n-t_k)}^{(1-\alpha)}\\
&\hspace{.4cm}+\frac{2}{(2-\alpha)}\big\{(\tau_{k-1}+\tau_{k-2}-\tau_k)(t_n-t_{k-1})^{2-\alpha}-(2\tau_k+\tau_{k-1}+\tau_{k-2}){(t_n-t_{k})^{2-\alpha}}\big\}\\
&\hspace{.4cm}+\frac{6}{(2-\alpha)(3-\alpha)}\big\{{(t_n-t_{k-1})}^{3-\alpha}-{(t_n-t_k)}^{3-\alpha}\big\}\Big],\\
w_{4,k}^{n} &= 
A_4\Big[\tau_{k-1}\tau_k{(t_n-t_{k-1})}^{(1-\alpha)}+\tau_k(\tau_k+\tau_{k-1}){(t_n-t_k)}^{1-\alpha}\\
&\hspace{.4cm}+\frac{2}{(2-\alpha)}\big\{{(\tau_k-\tau_{k-1})}{(t_n-t_{k-1})^{2-\alpha}}+(2\tau_k+\tau_{k-1})(t_n-t_k)^{2-\alpha}\big\}\\
&\hspace{.4cm}-\frac{6}{(2-\alpha)(3-\alpha)}\big\{{(t_n-t_{k-1})}^{3-\alpha}-{(t_n-t_k)}^{3-\alpha}\big\}\Big],
\end{aligned}
\end{equation*}
and
\begin{equation*}
\begin{aligned}
\begin{cases}\vspace{.2cm}
A_1 &= ~\frac{1}{\tau_k({\tau_k}+{\tau_{k-1}})({\tau_k}+{\tau_{k-1}}+{\tau_{k-2}})},\\\vspace{.2cm}
A_2 &= ~\frac{1}{\tau_k\tau_{k-1}(\tau_{k-1}+\tau_{k-2})},\\\vspace{.2cm}
A_3 &= ~\frac{1}{\tau_{k-1}\tau_{k-2}({\tau_{k}}+{\tau_{k-1}})},\\
A_4 &= ~\frac{1}{\tau_{k-2}({\tau_{k-1}}+{\tau_{k-2}})({\tau_k}+{\tau_{k-1}}+{\tau_{k-2}})}.\\
\end{cases}
\end{aligned}
\end{equation*}
\vspace{.1cm}

Now adjoining \eqref{e2}, \eqref{e3}, and \eqref{e4}, we obtain an approximation of the Caputo derivative of order \vspace{-1cm}$\alpha\in(0,1)$ as below
\begin{align}\notag
\textsubscript{C}D_{0,t}^{\alpha}u(x_i,t_n) &= \frac{1}{\Gamma(1-\alpha)}\int_{0}^{t_{n}}(t_{n}-s)^{-\alpha}{\frac{\partial u(x_i,s)}{\partial {s}}}ds, \\\notag
&= \frac{1}{\Gamma(1-\alpha)}\bigg[{\int_{t_0}^{t_1}}({{t_n}-s})^{-\alpha}{\frac{\partial u(x_i,s)}{\partial {s}}}ds + {\int_{t_1}^{t_2}}({{t_n}-s})^{-\alpha}{\frac{\partial u(x_i,s)}{\partial {s}}}ds \\\notag
&\hspace{2cm} +\sum_{j=3}^{n}{\int_{t_{j-1}}^{t_j}}({{t_n}-s})^{-\alpha}{\frac{\partial u(x_i,s)}{\partial {s}}}ds\bigg],\\\label{caputo_app}
&\approx \frac{1}{\Gamma(2-\alpha)} \sum_{j=0}^{n}{p_j}{U_i^{n-j}},\\\notag
&= \textsubscript{C}D_{0,t_{K_t}}^{\alpha}u(x_i,t_n).
\end{align}
Here the coefficients $p_j$ have different values for different $n$,
\begin{equation*}
\begin{aligned}
&\text{for }n=1,\ p_0=a_{1},\ \ \ p_1=-a_{1};\\
&\text{for }n=2,\ p_0=d_{2},\ \ \ p_1=a_{2}+c_{2},\ \ \ \ \ \ p_2=b_{2}-a_{2};\\
&\text{for }n=3,\ p_0=w_{1,3}^{3},\ p_1=w_{2,3}^{3}+d_{3},\ \ \ p_2=w_{3,3}^{3}+a_{3}+c_{3},\ \ \ \hspace{.3cm}p_3=w_{4,3}^{3}+b_{3}-a_{3};\\
&\text{for }n=4,\ p_0=w_{1,4}^{4},\ p_1=w_{1,3}^{4}+w_{2,4}^{4},\ \hspace{.01cm}  p_2=w_{3,4}^{4}+w_{2,3}^{4}+d_{4},\ \ \ p_3=w_{3,3}^{4}+w_{4,4}^{4}+a_{4}+c_{4},\\ &\hspace{1.7cm}p_4=w_{4,3}^{4}+b_{4}-a_{4};\\
&\text{for }n=5,\ p_0=w_{1,5}^{5},\ p_1=w_{1,4}^{5}+w_{2,5}^{5},\ p_2=w_{1,3}^{5}+w_{2,4}^{5}+w_{3,5}^{5},\ p_3=w_{2,3}^{5}+w_{3,4}^{5}+w_{4,5}^{5}+d_{5},\\&\hspace{1.7cm}p_4=w_{3,3}^{5}+w_{4,4}^{5}+a_{5}+c_{5},\ \ \ p_5=w_{4,3}^{5}+b_{5}-a_{5};
\end{aligned}
\end{equation*}
\ \ for $6\leq n\leq K_t$, the relation between the coefficients takes form as follows:
\begin{equation}\label{caputo_form}
\begin{aligned}
&\hspace{.3cm}p_0=w_{1,n}^{n},\ \ p_1=w_{1,n-1}^{n}+w_{2,n}^{n},\ \ p_{2} = w_{1,n-2}^{n}+w_{2,n-1}^{n}+w_{3,n}^{n},\hspace{2cm}\\
&\hspace{.3cm}p_{k} = w_{1,n-k}^{n}+w_{2,n-k+1}^{n}+w_{3,n-k+2}^{n}+w_{4,n-k+3}^{n}, \ (3\leq k\leq n-3),\\
&\hspace{.3cm}p_{n-2} = w_{2,3}^{n}+w_{3,4}^{n}+w_{4,5}^{n}+d_{n},\\
&\hspace{.3cm}p_{n-1} = w_{3,3}^{n}+w_{4,4}^{n}+a_{n}+c_{n},\ \ \ \hspace{.2cm} p_{n} = w_{4,3}^{n}+b_{n}-a_{n}. 
\end{aligned}
\end{equation}
\noindent

\subsection{\textbf{Truncation error estimate}}\noindent
This portion covers the estimation of truncation error for the approximation to the Caputo fractional derivative.
\begin{lemma}\label{trunclemma}
Let $u(\cdot,t)\in C[0,T]\cap C^4(0,T]$. Suppose that $\left|\frac{\partial^l u(\cdot,t)}{\partial t^l} \right|\leq C(1+t^{\alpha-l})$ for $l=0,1,\ldots,4,$ there exists a constant C such that $\forall$ $(x_i,t_n)\in \Omega_{h}\times \bar\Omega_{\tau},$ we have
\begin{equation}\label{trunc_eqn}
\begin{aligned}
\left|\textsubscript{C}D_{0,t_{K_t}}^{\alpha}u(x_i,t_n) -\textsubscript{C}D_{0,t}^{\alpha}u(x_i,t_n)\right|\leq Cn^{-\min\{4-\alpha,r\alpha\}},\\
\end{aligned}
\end{equation}
where C is a generic positive constant independent of $\tau_n$.
\end{lemma}
\begin{proof}
According to Taylor's expansion formula,
\begin{align}\notag
\tau_{k+1} = t_{k+1}-t_k &= T\Big(\frac{k+1}{K_t}\Big)^r-T\Big(\frac{k}{K_t}\Big)^r,\\\label{s1}
&\leq CT{K_t}^{-r}k^{r-1},\ k=0,1,\ldots,K_t-1.
\end{align}
For $n=1,2,\ldots,K_t, \text{ and }i=1,2,\ldots,K_x-1,$ we have
\begin{align}\notag
&\textsubscript{C}D_{0,t_{K_t}}^{\alpha}u_i^n-\textsubscript{C}D_{0,t}^{\alpha}u(x_i,t_n) = \sum_{k=1}^nT_{nk},\\\label{s2}
&\left|\textsubscript{C}D_{0,t_{K_t}}^{\alpha}u_i^n-\textsubscript{C}D_{0,t}^{\alpha}u(x_i,t_n)\right| \leq \sum_{k=1}^n\left|T_{nk}\right|.
\end{align}
According to our algorithm, we need to form at least three time subintervals to discretize the Caputo derivative,  i.e., $n\geq3$.
Here, for $k = 1$ in \eqref{s2}, we get
\begin{equation*}
\begin{aligned}
T_{n1} = \frac{1}{\Gamma(1-\alpha)}{\int_{t_0}^{t_1}}({{t_n}-s})^{-\alpha}\bigg[{\frac{\partial I_{11}u}{\partial s}(x_i,s)-\frac{\partial u}{\partial s}(x_i,s)}\bigg]ds.
\end{aligned}
\end{equation*}
Moreover, from \cite{Stynes_SIAM}, we have
\begin{equation}\label{s3}
\begin{aligned}
\left|T_{n1}\right|\leq Cn^{-r\alpha},\text{ where } n=3,4,\ldots,K_t.
\end{aligned}
\end{equation}
Now for $k = 2$,
\begin{equation*}
\begin{aligned}
T_{n2} = \frac{1}{\Gamma(1-\alpha)}{\int_{t_1}^{t_2}}({{t_n}-s})^{-\alpha}\bigg[{\frac{\partial I_{22}u}{\partial s}(x_i,s)-\frac{\partial u}{\partial s}(x_i,s)}\bigg]ds.
\end{aligned}
\end{equation*}
From \cite{qiao2022fast}, we have
\begin{equation}\label{s4}
\begin{aligned}
\left|T_{n2}\right|\leq Cn^{-r(1+\alpha)}, \text{ where }n=3,4,\ldots,K_t.
\end{aligned}
\end{equation}
For $k = 3,4,\ldots,K_t,$
\begin{equation*}
\begin{aligned}
T_{nk} = \frac{1}{\Gamma(1-\alpha)}{\int_{t_{k-1}}^{t_k}}({{t_n}-s})^{-\alpha}\bigg[{\frac{\partial I_{3k}u}{\partial s}(x_i,s)-\frac{\partial u}{\partial s}(x_i,s)}\bigg]ds.
\end{aligned}
\end{equation*}
For $n=3,4,\ldots,K_t,\text{ and } k=3,4,\ldots,\lceil{\frac{n-1}{2}}\rceil$, using \eqref{int_error} here we get,
\begin{align*}
T_{nk} &= \frac{1}{\Gamma(1-\alpha)}{\int_{t_{k-1}}^{t_k}}({{t_n}-s})^{-\alpha}\bigg[{\frac{\partial I_{3k}u}{\partial s}(x_i,s)-\frac{\partial u}{\partial s}(x_i,s)}\bigg]ds,\\
&= \frac{-\alpha}{\Gamma(1-\alpha)}{\int_{t_{k-1}}^{t_k}}({{t_n}-s})^{-\alpha-1}\frac{u_{tttt}(x_i,\phi_k)}{12}(s-t_{k-3})(s-t_{k-2})(s-t_{k-1})(s-t_{k})ds.\\
\end{align*}
Then, using the assumption $\left|\frac{\partial^l u(\cdot,t)}{\partial t^l} \right|\leq C(1+t^{\alpha-l})$ for $l=0,1,\ldots,4$, we get
\begin{equation}\label{s5}
    \left|T_{nk}\right| \leq \ Ct_{k-1}^{\alpha-4}\tau_{k}^{4}\int_{t_{k-1}}^{t_k}(t_n-s)^{-\alpha-1}ds \ \leq Ct_{k-1}^{\alpha-4}\tau_{k}^{5}(t_n-t_k)^{-\alpha-1},
\end{equation}
where
\begin{align}\notag
(t_n-t_k)^{-\alpha-1}={\bigg(T\frac{n^r-k^r}{{K_t}^r}\bigg)}^{-\alpha-1}&=T^{-\alpha-1}\bigg({\frac{{K_t}^r}{n^r-k^r}}\bigg)^{\alpha+1},\\\notag
&\leq T^{-\alpha-1}\bigg({\frac{{K_t}^r}{n^r-{\lceil\frac{n-1}{2}\rceil}^r}}\bigg)^{\alpha+1},\\\label{s6} 
&\leq CT^{-\alpha-1}{\bigg(\frac{K_t}{n}\bigg)}^{r(1+\alpha)}.
\end{align}
Now, using \eqref{s1} and \eqref{s6} in \eqref{s5}, we get
\begin{align*} 
\left|T_{nk}\right| &\leq C\bigg[T\bigg(\frac{k-1}{K_t}\bigg)^r\bigg]^{\alpha-4}T^5{K_t}^{-5r}(k-1)^{5(r-1)}T^{-\alpha-1}\left(\frac{K_t}{n}\right)^{r(1+\alpha)},\\
&\leq C{K_t}^{r(4-\alpha)-5r+r(1+\alpha)}n^{-r(1+\alpha)}(k-1)^{r(\alpha-4)+5(r-1)},\\
&\leq Cn^{-r(1+\alpha)}(k-1)^{r(1+\alpha)-5}.
\end{align*}
Hence,
\begin{align}\notag
&\sum_{k=3}^{\lceil\frac{n-1}{2}\rceil}T_{nk}\leq Cn^{-r(1+\alpha)}\sum_{k=3}^{\lceil\frac{n-1}{2}\rceil}(k-1)^{r(1+\alpha)-5},\\\label{s8}
&\hspace{1.55cm}\leq
\begin{cases}
Cn^{-r(1+\alpha)},\ \ \ r(1+\alpha)<4,\\
Cn^{-4}\ln (n),\ \ r(1+\alpha)=4,\\
Cn^{-4},\ \ \ \ \ \ \ \ \ r(1+\alpha)>4,
\end{cases}
\end{align}
and for $k=\lceil{\frac{n-1}{2}}\rceil+1,\ldots,n$, we have
\begin{align}\notag
\sum_{k=\lceil\frac{n-1}{2}\rceil+1}^{n-1}T_{nk}&=\sum_{k=\lceil\frac{n-1}{2}\rceil+1}^{n-1}t_{k-1}^{\alpha-4}\tau_k^4\int_{t_{k-1}}^{t_k}(t_n-s)^{-\alpha-1}ds,\\\notag
&\leq Ct_{k-1}^{\alpha-4}\tau_k^4\sum_{k=\lceil\frac{n-1}{2}\rceil+1}^{n-1}\frac{1}{\alpha}\bigg[(t_n-t_k)^{-\alpha}-(t_n-t_{k-1})^{-\alpha}\bigg],\\\notag
&\leq Ct_{k-1}^{\alpha-4}\tau_k^4\bigg[(t_n-t_{n-1})^{-\alpha}-(t_n-t_{\lceil\frac{n-1}{2}\rceil})^{-\alpha}\bigg],\\\notag
&\leq Ct_{k-1}^{\alpha-4}\tau_k^4\tau_n^{-\alpha}\leq Ct_{k-1}^{\alpha-4}\bigg[TK_t^{-r}(k-1)^{r-1}\bigg]^4\tau_n^{-\alpha},\\\notag
&\leq Ct_n^{\alpha-4}T^4K_t^{-4r}n^{4(r-1)}\bigg[TK_t^{-r}n^{r-1}\bigg]^{-\alpha},\\\notag
&\leq CT^4K_t^{-4r}n^{4(r-1)}T^{\alpha-4}\bigg(\frac{n}{K_t}\bigg)^{r(\alpha-4)}T^{-\alpha}K_t^{r\alpha}n^{-\alpha(r-1)},\\\label{s9}
&\leq Cn^{-(4-\alpha)}.
\end{align}
Finally, we consider $T_{nn}$ for $n>2$ as
\begin{align}\notag
T_{nn} &= \frac{1}{\Gamma(1-\alpha)}{\int_{t_{n-1}}^{t_n}}({{t_n}-s})^{-\alpha}\bigg[{\frac{\partial I_{3k}u}{\partial s}(x_i,s)-\frac{\partial u}{\partial s}(x_i,s)}\bigg]ds,\\\notag
& \leq Ct_{n-1}^{\alpha-4}\tau_{n-1}^{3}\int_{t_{n-1}}^{t_n}({{t_n}-s})^{-\alpha}ds,\\\notag
& \leq Ct_{n-1}^{\alpha-4}\tau_{n-1}^{3}({t_n}-t_{n-1})^{1-\alpha},\\\notag
& \leq C\big[TK_t^{-r}(n-1)^r\big]^{\alpha-4}\big[TK_t^{-r}(n-2)^{r-1}\big]^{3}\big[TK_t^{-r}(n-1)^{r-1}\big]^{1-\alpha},\\\label{s10}
& \leq Cn^{-(4-\alpha)}. 
\end{align}
Combining formulas \eqref{s3}, \eqref{s4}, and \eqref{s8}-\eqref{s10}, we obtain
\begin{equation*}
\begin{aligned}
\left|\textsubscript{C}D_{0,t_{K_t}}^{\alpha}u(x_i,t_n)-\textsubscript{C}D_{0,t}^{\alpha}u(x_i,t_n)\right| \leq Cn^{-\min\{4-\alpha,r\alpha\}}.
\end{aligned}
\end{equation*}
\end{proof}
\begin{remark}\label{rem_1}
    We get the truncation error of the approximation to Caputo derivative of a non-smooth function on the graded mesh as $\mathcal{O}(\tau^{\min\{4-\alpha,r\alpha\}})$. However, in \cite{Cao_2015}, the truncation error of Caputo derivative approximation for a smooth function on the uniform mesh is obtained as $\mathcal{O}(\tau^{4-\alpha})$. Hence it can be easily observed that for a non-smooth function, grading improves the convergence of the approximation and for $r\geq\frac{4-\alpha}{\alpha}$, the convergence will always be optimum and equal to that of the smooth solution, i.e., $(4-\alpha)$. 
\end{remark}
In this next section, with the help of the above approximation of the Caputo derivative, we construct a high-order numerical scheme for TFDE \eqref{diff_eqn}-\eqref{ini_eqn} and perform some important analysis on it.
\section{Numerical scheme for TFDE}\label{scheme}
\noindent
For the function $u(x,t)$, its analytical and approximate solutions at the mesh point $(x_i,t_n)$ are denoted by $U_i^n$ and $u_i^n$, respectively.
In view of the obtained results \eqref{space_eqn} and \eqref{trunc_eqn}, we consider the following approximation of the TFDE \eqref{diff_eqn}:
\begin{equation}\label{h1}
\begin{aligned}
L_{K_x,K_t}u_i^n=\textsubscript{C}D_{0,t_{K_t}}^{\alpha}u_i^n-\varrho\delta_x^2u_i^n=f_i^n,\ 1\leq i\leq K_x-1, 1\leq n\leq K_t,
\end{aligned}
\end{equation}
Therefore, we derive the following finite difference scheme from \eqref{h1}:
\begin{equation}\label{h3}
\begin{aligned}
&\frac{1}{\Gamma(2-\alpha)}\Big[{p_0}u_i^n+{p_1}u_i^{n-1}+{p_2}u_i^{n-2}+\sum_{k=3}^{n-3}{p_k}u_i^{n-k}+{p_{n-2}}u_i^2+{p_{n-1}}u_i^1+{p_n}u_i^0\Big]\\ 
&\hspace{1.5cm}= \varrho \frac{u_{i+1}^n-2u_i^n+u_{i-1}^n}{h^2}+f_i^n,
\end{aligned}
\end{equation}
where $1\leq n\leq K_t$, and $1\leq i\leq K_x-1$. The initial and boundary conditions \eqref{bdry_eqn}-\eqref{ini_eqn} can be discretized as
\begin{align}\label{h4}
u_0^n &= 0 = u_{K_x}^n,\ 0< n\leq K_t,\\\label{h5}
u_i^0 &= \varphi(x_i), \ 0\leq i\leq K_x.
\end{align}
Multiplying \eqref{h3} by $h^2$ and introducing the parameter $\mu=\frac{h^2}{\Gamma(2-\alpha)}$ to simplify the expression, we get
\begin{equation}\label{h6}
\begin{aligned}
\text{ for } n=1,\ -\varrho u_{i-1}^1+(\mu p_0+2\varrho)u_i^1-\varrho u_{i+1}^1 &= -\mu{p_1}u_i^0+h^2{f_i^1};\\
\text{ for } n=2,\ -\varrho u_{i-1}^2+(\mu p_0+2\varrho)u_i^2-\varrho u_{i+1}^2 &= -\mu{p_1}u_i^1-\mu{p_2}u_i^0+h^2{f_i^2};\\
\text{ for } n\geq 3,\ -\varrho u_{i-1}^n+(\mu p_0+2\varrho)u_i^n-\varrho u_{i+1}^n  &= -\mu\sum_{k=1}^n{p_k}u_i^{n-k}+h^2{f_i^n}.\\
\end{aligned}
\end{equation}
The system of linear equations \eqref{h6} can also be rewritten in the following matrix form:
\begin{align*}
\text{ for } n=1,\ {A}u^1 &= -\mu p_1{u^0}+h^2{F^1};\\
\text{ for } n=2,\ {A}u^2 &= -\mu p_1{u^1}-\mu p_2{u^0}+h^2{F^2};\\
\text{ for } n\geq3,\ {A}u^n &= -\mu{\sum_{k=1}^n} p_k{u^{n-k}}+h^2{F^n};
\end{align*}
where the matrix and vectors are defined by
\begin{align}\label{A}
A &=\ tri\big[-\varrho,\mu {p_0}+2\varrho,-\varrho \big],\\\label{u}
u^n &=\ (u_1^n,u_2^n,\ldots,u_{K_x-1}^n)^T,\\ \label{f}
F^n &=\ (f_1^n,f_2^n,\ldots,f_{K_x-1}^n)^T,\text{ for } 1\leq n\leq K_t.
\end{align}

\begin{theorem}
The solution of the difference scheme \eqref{h3}-\eqref{h5} is unique. 
\end{theorem}
\begin{proof}
The coefficient matrix $A$ in \eqref{A} is strictly diagonally dominant irrespective of any possible values of $h,\tau_n$, and $\alpha$, so it is non-singular. Hence, $A$ is invertible. Consequently, the difference scheme \eqref{h3}-\eqref{h5} has a unique solution.
\end{proof}
\subsection{\textbf{Stability and convergence analysis of the scheme on uniform mesh}}\label{stability}
\noindent
In this segment, we discuss the stability and convergence of the proposed difference scheme \eqref{h3}-\eqref{h5} on a uniform mesh. To accomplish that, we first rewrite the difference scheme with $r=1$ in a desired manner.\\
For the case $r=1$, the coefficients become $p_j=\tau^{-\alpha}g_j$. Hence, the scheme \eqref{h3} is now
\begin{align*}
    &\frac{\tau^{-\alpha}}{\Gamma(2-\alpha)}\left[\sum_{k=0}^{n}{g_k}u_i^{n-k}\right]= \varrho \frac{u_{i+1}^n-2u_i^n+u_{i-1}^n}{h^2}+f_i^n,\\
    &-\eta\epsilon u_{i+1}^n+(g_0+2\eta\epsilon )u_i^n-\eta\epsilon u_{i-1}^n=-\sum_{k=1}^{n}{g_k}u_i^{n-k}+\epsilon f_i^n,
\end{align*}
where $\eta=\frac{\varrho}{h^2}$ and $\epsilon=\tau^{\alpha}\Gamma(2-\alpha)$.
So the matrix equation of the above scheme is
\begin{equation}\label{me1}
    Tu^n = -{\sum_{k=1}^n} g_k{u^{n-k}}+\epsilon{F^n},
\end{equation}
where the matrix $T$ is given by 
\begin{equation*}
    T =\ tri\big[-\eta\epsilon,g_0+2\eta\epsilon ,-\eta\epsilon \big],
\end{equation*}
and other vectors $u^n$ and $F^n$ are from \eqref{u} and \eqref{f}, respectively.

\bigskip
Following is a result from \cite{Cao_2015}, which will be used to establish the stability and convergence of the proposed scheme on uniform mesh, i.e., $r=1$. 
\begin{lemma}\label{lem3.1}
For the case $r=1$, the following properties of the coefficients $g_j$ hold:
\begin{enumerate}
    \label{xx}\item[$(1)$]  When $n\geq 3$, for any order $\alpha \in (0,1)$,
    \bigskip
    \begin{enumerate}
        \item $g_0=\frac{1}{3}+\frac{1}{2-\alpha}+{\frac{1}{(2-\alpha)(3-\alpha)}}\in(1,\frac{11}{6})$,
        \item $g_2>0$,\vspace{.1cm}
        \item $g_j<0,\ j\geq1,j\neq2$.
    \end{enumerate}\bigskip
    \item[$(2)$]$\sum_{j=0}^n{g_j}=0$, for any order $\alpha \in (0,1)$ and $\text{for all } n$.
\end{enumerate}
\end{lemma}
\begin{proof}
    For proof, we refer \cite[Lemma 2.1 and 2.2]{Cao_2015}
\end{proof}
Next, we derive a lemma that will be further utilized to obtain the stability estimate of the scheme.
\begin{lemma} The norm $\|T^{-1}\|_\infty \leq 1 $ where $\|T\|_\infty=\max_{1\leq i\leq K_x-1}\Big\{\sum_{j=1}^{K_x-1}|t_{i,j}|\Big\}$.
\end{lemma}

\begin{proof}
    The above-given statement can be proved with the help of the Gerschgorin theorem. It states that all the eigenvalues of the matrix $T$ should lie in the union of $K_x-1$ circles centered at $t_{i, i}$ with radius $r_i = \sum^{K_x-1}_{k=1,k\neq i} |t_{i,k}|$. We need to show that every eigenvalue of the matrix $T$ has a magnitude greater than 1. Hence, using the Gerschgorin theorem and the properties of matrix $T$, we have
    \begin{align*}
        &t_{i,i}=g_0+2\eta\epsilon,\ \ r_i=\sum^{K_x-1}_{k=1,k\neq i} |t_{i,k}|\leq 2\eta\epsilon,\\
        \text{ and so }\ &|\lambda-(g_0+2\eta\epsilon)|\leq 2\eta\epsilon.
    \end{align*}
    Since for $r=1$ from Lemma \ref{lem3.1}, we have $g_0\in(1,11/6)$, and therefore we conclude that the eigenvalues of the matrix $T$ have a magnitude greater than $1$. Also, the matrix $T$ is diagonally dominant. Thus $T$ is invertible and the eigenvalues of $T^{-1}$ are less than or equal to $1$ in magnitude. Therefore,
    \begin{equation*}
        \|T^{-1}\|_\infty\leq 1.
    \end{equation*}
\end{proof}
\noindent
Now the stability of the proposed difference scheme derived above is proved in the following theorem with the help of this lemma.

\begin{theorem}
    The proposed finite difference scheme defined by \eqref{h3}-\eqref{h5} to the TFDE \eqref{diff_eqn}-\eqref{ini_eqn} for $r=1$ is unconditionally stable.
\end{theorem}

\begin{proof}
   Let $u_i^n$ and $\Tilde{u}_i^n$ be the approximate solutions of \eqref{h3} corresponding to different initial conditions. Therefore, $v_i^n=u_i^n-\Tilde{u}_i^n$ be another solution of \eqref{h3} where $v_i^0\neq 0$ and $v^n=(v_1^n,v_2^n,\ldots,v_{K_x-1}^n)^T$. The difference solution $v^n$ will also satisfy \eqref{me1} with $F^n=\textbf{0}$.
   \begin{equation*}
    Tv^n = -{\sum_{k=1}^n} g_k{v^{n-k}}.
\end{equation*}
From the technique of mathematical induction, we prove $\|v^n\|\leq C\|v^0\|,\ n=0,1,2,\ldots,K_t$, where C is some positive constant. Hence for $n=1$ in the above equation, we have
    \begin{align*}
        Tv^1&=-g_1v^0,\\
        \|v^1\|_\infty&\leq \left\|T^{-1}\right\|_\infty\ \left\|\left(-g_1v^0\right)\right\|_\infty,\\
        &\leq |-g_1|\ \left\|v^0\right\|_\infty.
        \end{align*}
        For $n=1$, we have $g_1=-g_0$ and $g_0=1$. So
         \begin{align*}
       \left\|v^1\right\|_\infty&\leq |g_0|\ \left\|v^0\right\|_\infty,\\
        \left\|v^1\right\|_\infty&\leq \left\|v^0\right\|_\infty.
    \end{align*}
   Now assume that $\|v^s\|_\infty\leq C\|v^0\|_\infty,\ \forall s\leq n$. We will prove it is also true for $s=n+1$. We have,
    \begin{align*}
     Tv^{n+1} &= -{\sum_{k=1}^{n+1}}
     g_k{v^{n+1-k}},\\
\left\|v^{n+1}\right\|_\infty&=\left\|T^{-1} \left(-\sum_{k=1}^{n+1}{g_k}v^{n+1-k}\right)\right\|_\infty,\\
    &\leq \left\|T^{-1}\right\|_\infty\  \left\|\left(-\sum_{k=1}^{n+1}{g_k}v^{n+1-k}\right)\right\|_\infty,\\
    &\leq \sum_{k=1}^{n+1}|{g_k}| \ \left\|v^{n+1-k}\right\|_\infty,\\
    &\leq C\left\|v^0\right\|_\infty\ \left(\sum_{k=1}^{n+1}|{g_k}|\right),
    \end{align*}
Using the identities $(1)$ and $(2)$ of Lemma \ref{lem3.1}, we get
    \begin{align*}
   \left\|v^{n+1}\right\|_\infty &\leq C(g_0+2g_2) \left\|v^0\right\|_\infty,\\
    &\leq C_1\left\|v^0\right\|_\infty.
    \end{align*}
    Hence the scheme is unconditionally stable.
\end{proof}

\begin{theorem}
  The solutions $u(x_i,t_n)$ of the TFDE \eqref{diff_eqn}-\eqref{ini_eqn} and $u_i^n$ of the difference scheme \eqref{h3}-\eqref{h5} at grid point $(x_i,t_n)$, for $r=1$ and sufficiently small $h$ and $\tau$, satisfy
 \begin{equation}
 \|e^n\|_\infty= \mathcal{O}\big({h^2}+\tau^{\alpha}\big),\ n=1,2,\ldots,K_t,
 \end{equation}
 where $e_i^n = u(x_i,t_n)-u_i^n$, $i=1,2,\ldots,K_x-1$.  
\end{theorem}

\begin{proof}
   Let us fix $(x_i,t_n)\in \Omega$. Using \eqref{h1}, \eqref{h4}-\eqref{h5}, we obtain the following difference equation for the error
 \begin{equation*}
 \begin{aligned}
 &L_{K_x,K_t}e_i^n=r_i^n,~1\leq i\leq K_x-1,\ 1\leq n\leq K_t,\\
 &e_0^n = 0 = e_{K_x}^n,~0<n\leq K_t,\\
 &e_i^0 = 0,~0\leq i\leq K_x-1.
 \end{aligned}
 \end{equation*}
 We have from our error equation,  \eqref{space_eqn} and Lemma \ref{trunclemma},
 \begin{equation*}
 \begin{aligned}
|r_i^n|=|L_{K_x,K_t}e_i^n|&=\big|\textsubscript{C}D_{0,K_t}^{\alpha}e_i^n-\delta_x^2e_i^n\big|,\\
 &=\Big|\textsubscript{C}D_{0,K_t}^{\alpha}\left[u(x_i,t_n)-u_i^n\right]-\delta_x^2\left[u(x_i,t_n)-u_i^n\right]\Big|,\\
 &\hspace{.1cm}\leq C\big({h^2}+n^{-\min\{4-\alpha,r\alpha\}}\big),
 \end{aligned}
 \end{equation*} 
For $r=1$, it will convert into
 \begin{equation*}
 \begin{aligned}
 |r_i^n|&\leq C\big({h^2}+n^{-\min\{4-\alpha,\alpha\}}\big),\\
 &= C\big({h^2}+n^{-\alpha}\big).
 \end{aligned}
 \end{equation*}
Now for $n=1$, we have from error equation
\begin{align*}
     &-\eta\epsilon e_{i+1}^1+(g_0+2\eta\epsilon)e_i^1-\eta\epsilon e_{i-1}^1= r_i^1,
\end{align*}
And for $n>1$
\begin{equation*}
     -\eta\epsilon e_{i+1}^n+(g_0+2\eta\epsilon)e_i^n-\eta\epsilon e_{i-1}^n+\sum_{k=1}^{n}{g_k}e_i^{n-k}= r_i^n,
\end{equation*}
We use the mathematical induction method to prove the convergence of the scheme. If $n=1$, we choose $|e_l^1|=\max_{1\leq i\leq K_x-1} |e_i^1|$ and hence,
\begin{align*}
   |e_l^1|&\leq (g_0+2\eta\epsilon) |e_l^1|-\eta\epsilon|e_{l+1}^1|-\eta\epsilon|e_{l-1}^1|,\\
   &\leq |(g_0+2\eta\epsilon) e_l^1-\eta\epsilon e_{l+1}^1-\eta\epsilon e_{l-1}^1|,\\
   &=| r_l^1| \leq C({h^2}+\tau^{\alpha}),\\
   \|e^1\|_\infty&\leq C({h^2}+\tau^{\alpha}).
\end{align*}
Suppose that if $s\leq n-1$, $||e^s||_\infty\leq M({h^2}+\tau^{\alpha})$ holds, then for $s=n$, let $|e_l^{n}|=\max_{1\leq i\leq K_x-1} |e_i^{n}|$, we have

\begin{align*}
    |e_l^{n}|&\leq (g_0+2\eta\epsilon) |e_l^{n}|-\eta\epsilon|e_{l+1}^{n}|-\eta\epsilon|e_{l-1}^{n}|,\\
   &\leq \left|(g_0+2\eta\epsilon) e_l^{n}-\eta\epsilon e_{l+1}^{n}-\eta\epsilon e_{l-1}^{n}\right|,\\
   &\leq \left|-\sum_{k=1}^{n}{g_k}e_l^{n-k}+ r_l^{n}\right|,\\
   &\leq \sum_{k=1}^{n}|{g_k}||e_l^{n-k}|+ |r_l^{n}|,\\
   &\leq M_1({h^2}+\tau^{\alpha})\left(\sum_{k=1}^{n}|{g_k}|+1\right).
   \end{align*}
   Similar as above, using the identities $(1)$ and $(2)$ of Lemma \ref{lem3.1}, we get
\begin{align*}
    |e_l^{n}|&\leq (1+g_0+2g_2)M_1({h^2}+\tau^{\alpha}).
\end{align*}
Thus $\|e^{n}\|_\infty\leq  M_0\big({h^2}+\tau^{\alpha}\big)$.
\end{proof}

\begin{remark}\label{rem_2}
As discussed in Remark \ref{rem_1}, we know from \cite[Theorem 4.1]{Cao_2015} that for the TFDE problem \eqref{diff_eqn}-\eqref{ini_eqn} having a smooth solution, i.e., $u(\cdot,t)\in C^4[0,T]$ on uniform mesh, the local truncation error for the approximation of Caputo fractional derivative is $\mathcal{O}(\tau^{4-\alpha})$. 
Therefore, keeping in view the proof of the above theorem, one can deduce the rate of convergence for the proposed scheme \eqref{h3}-\eqref{h5} in case of smooth solution on a uniform mesh as $\mathcal{O}(\tau^{4-\alpha}+h^2)$. Also, the stability result in \cite{Cao_2015} is proved via the Fourier series method for $\alpha\in(0,\bar{\bar{\bar{\alpha}}}\approx0.368)$, whereas we have shown the stability through the matrix method for any $\alpha\in (0,1)$.
\end{remark}

\begin{remark}\label{rem_3}
The above stability and convergence results for the difference scheme \eqref{h3}-\eqref{h5} have been proved for $r=1$ i.e., on a uniform mesh. The rationale behind it is that the relation between the coefficients $p_j$ for $r=1$, i.e. $g_j$, is known, but for $r>1$, the behavior of the coefficients becomes very complex. From the table below we can see that the signum behavior of the coefficients fluctuates for different $r$. For the same number of iterations, the number of positive coefficients varies for different values of $r$ which indicates some of the coefficients are changing their signs. This behavioral change also depends on $n$ as the change in the number of iterations for a fix $r$ changes the number of coefficients of a particular sign. Thus, the behavior of coefficients alters for different $n$ and $r$ for a fixed value of $\alpha$ (here $\alpha=0.5$) and hence we are unable to conclude about the stability of the scheme for $r>1$.
\end{remark}

\begin{table}[htbp]
\begin{center}
\begin{tabular}{c|c c|c c|c c}
\hline
 &\multicolumn{2}{|c|}{$r=4$} &\multicolumn{2}{|c|}{$r=6$} &\multicolumn{2}{|c}{$r=10$}\\
\cline{2-7} 
\textbf{$n$} & +ve sign & -ve sign & +ve sign & -ve sign & +ve sign & -ve sign\\

\hline
$10$ & $4$ & $7$ & $5$ & $6$ & $4$ &  $7$\\
$20$ & $4$ & $17$  & $5$ & $16$ & $6$ & $15$ \\
$30$ & $3$ & $28 $ & $6$ & $25$ & $9$ & $22$ \\
$40$ & $4$ & $37$ & $7$ & $34$ & $12$ & $29$\\
$50$ & $6$ & $45$ & $9$ & $42$ & $16$ & $35$\\
\end{tabular}
\caption{Number of +ve and -ve coefficients $(p_j)$ for $N=50$ and $\alpha=0.5$ for different values of $n$.}
\label{table:9}
\end{center}
\end{table}


\section{Numerical experiments}\label{results}\noindent
In this section, we verify the theoretical analysis provided in the previous section by producing numerical results for two test examples of TFDE problem \eqref{diff_eqn}-\eqref{ini_eqn} with smooth and non-smooth solutions, respectively, using the proposed difference scheme \eqref{h3}-\eqref{h5}. We examine the accuracy and efficiency of the scheme for different values of $K_x$, $K_t$, $\alpha$, and $r$. The exact solution to each case is known in advance so that the results can be compared and the convergence rate is verified. The following error norm is used to check the accuracy of the proposed scheme:
\begin{equation}\label{err_norm}
E_\infty(K_t,K_x)=\max_{1\leq i\leq K_x-1}\left|u(x_i,t_{K_t})-u_i^{K_t}\right|.
\end{equation}
Using the error norm defined above, we denote the numerical convergence orders in time and space by $T_{rate}$ and $S_{rate}$, respectively, and are defined as
\begin{equation}\label{rate}
    T_{rate} = \log_2\bigg(\frac{E_\infty(K_t/2,K_x)}{E_\infty(K_t,K_x)}\bigg), \ \
S_{rate} = \log_2\bigg(\frac{E_\infty(K_t,K_x/2)}{E_\infty(K_t,K_x)}\bigg).
\end{equation}
Here, $K_x$ and $K_t$ are correspondingly the total numbers of subintervals in space and time directions.\par
We begin by analysing the error of the considered problem \eqref{diff_eqn}-\eqref{ini_eqn} through an example having a smooth solution, i.e., $u(\cdot,t)\in C^4[0,T]$.

\begin{example}\label{ex1}
Let us consider TFDE problem \eqref{diff_eqn} with the following data:
\begin{equation*}
\begin{aligned}
u&(0,t) = u(\pi,t) = 0,~~~0 < t \leq 1,\\
&u(x,0) = 0,~~~0\leq x\leq\pi.
\end{aligned}
\end{equation*}
Here $f(x,t)=t^5\sin x \big(1+\frac{120}{\Gamma(6-\alpha)}t^{-\alpha}\big)$ and thus, the exact solution is $u(x,t)= t^5\sin x$.
\end{example}\noindent
Clearly, the above example has a smooth solution throughout the time interval i.e., $u(\cdot,t)\in C^4[0,T]$. Therefore we have calculated the errors on a uniform mesh, i.e., $r=1$, using the relation \eqref{err_norm} at final time $T=1$. Numerical results are given in Table \ref{table:1} and \ref{table:2} for different numbers of temporal splits with two different values of diffusion coefficient $\varrho=1,30$; fixing $K_x=10000$. The results indicate that the maximum nodal errors are minimal and the convergence order of scheme \eqref{h1}-\eqref{h3} for a smooth solution is $(4-\alpha)$, which is the expected rate as discussed in Remark \ref{rem_2}. It is also observed that as the fractional order $\alpha$ increases, simultaneously, the magnitudes of the maximum errors increase and the convergence rate decreases. One can note in Table~\ref{table:2} that with a high diffusion coefficient, the errors are relatively smaller, which specifies that a higher diffusion coefficient fastens the physical process.

\begin{table}[htbp]
	\centering
	\begin{tabular}{c| c c| c c| c c}
		\hline
		&\multicolumn{2}{|c|}{$\alpha=0.3$} &\multicolumn{2}{|c|}{$\alpha=0.6$} &\multicolumn{2}{|c}{$\alpha=0.8$}\\
		\cline{2-7} 
		$K_t$  & $E_\infty$ & Order & $E_\infty$ & Order & $E_\infty$ & Order\\ \hline
		$10$ & $4.3693e-04$ &  & $1.9000e-03$ &  & $4.3000e-03$ & \\  
		$20$ & $3.8295e-05$ & $3.5122$ & $2.0232e-04$ & $3.2313$ & $5.2108e-04$ & $3.0448$\\ 
		$40$ & $3.1936e-06$ & $3.5839$ & $2.0257e-05$ & $3.3201$ & $5.9698e-05$ & $3.1258$\\ 
		$80$ & $2.6190e-07$ & $3.6081$  & $1.9790e-06$ & $3.3556$ & $6.6679e-06$ & $3.1624$\\  
		$160$ & $2.0187e-08$ & $3.6975$  & $1.9197e-07$ & $3.3658$ & $7.3596e-07$ & $3.1795$\\      
	\end{tabular}
	\caption{Maximum temporal errors and  order of convergence $(T_{rate})$ for $r=1$ and $\varrho=1$.}
	\label{table:1}
\end{table}
\begin{table}[htbp]
    \centering
    \begin{tabular}{c| c c| c c| c c}
     \hline
&\multicolumn{2}{|c|}{$\alpha=0.3$} &\multicolumn{2}{|c|}{$\alpha=0.6$} &\multicolumn{2}{|c}{$\alpha=0.8$}\\
\cline{2-7} 
   $K_t$   & $E_\infty$ & Order & $E_\infty$ & Order & $E_\infty$ & Order\\ \hline
   $10$ & $3.0835e-05$ &  & $1.4990e-04$ &  & $3.6707e-04$ & \\  
   $20$ & $2.6811e-06$ & $3.5237$ & $1.5458e-05$ & $3.2776$ & $4.3006e-05$ & $3.0934$\\ 
   $40$ & $2.3099e-07$ & $3.5369$ & $1.5369e-06$ & $3.3303$ & $4.8561e-06$ & $3.1467$\\ 
   $80$ & $1.9091e-08$ & $3.5969$  & $1.5544e-07$ & $3.3056$ & $5.4559e-07$ & $3.1539$\\  
   $160$ & $1.4818e-09$ & $3.6875$  & $1.5646e-08$ & $3.3125$ & $6.6199e-08$ & $3.0421$\\      
    \end{tabular}
    \caption{Maximum temporal errors and order of convergence $(T_{rate})$ for $r=1$ and $\varrho=30$.}
    \label{table:2}
\end{table}

 The graphical illustration of comparison between the exact solution and numerical solution of Example \ref{ex1} on uniform mesh with diffusion coefficient $\varrho=1$ and $\varrho=30$ is given in Figure \ref{fig1} and Figure \ref{fig2}, respectively.

\begin{figure}[htbp]
	\begin{center}
		\centering
		\subfigure[]{%
			\includegraphics[scale=0.35]{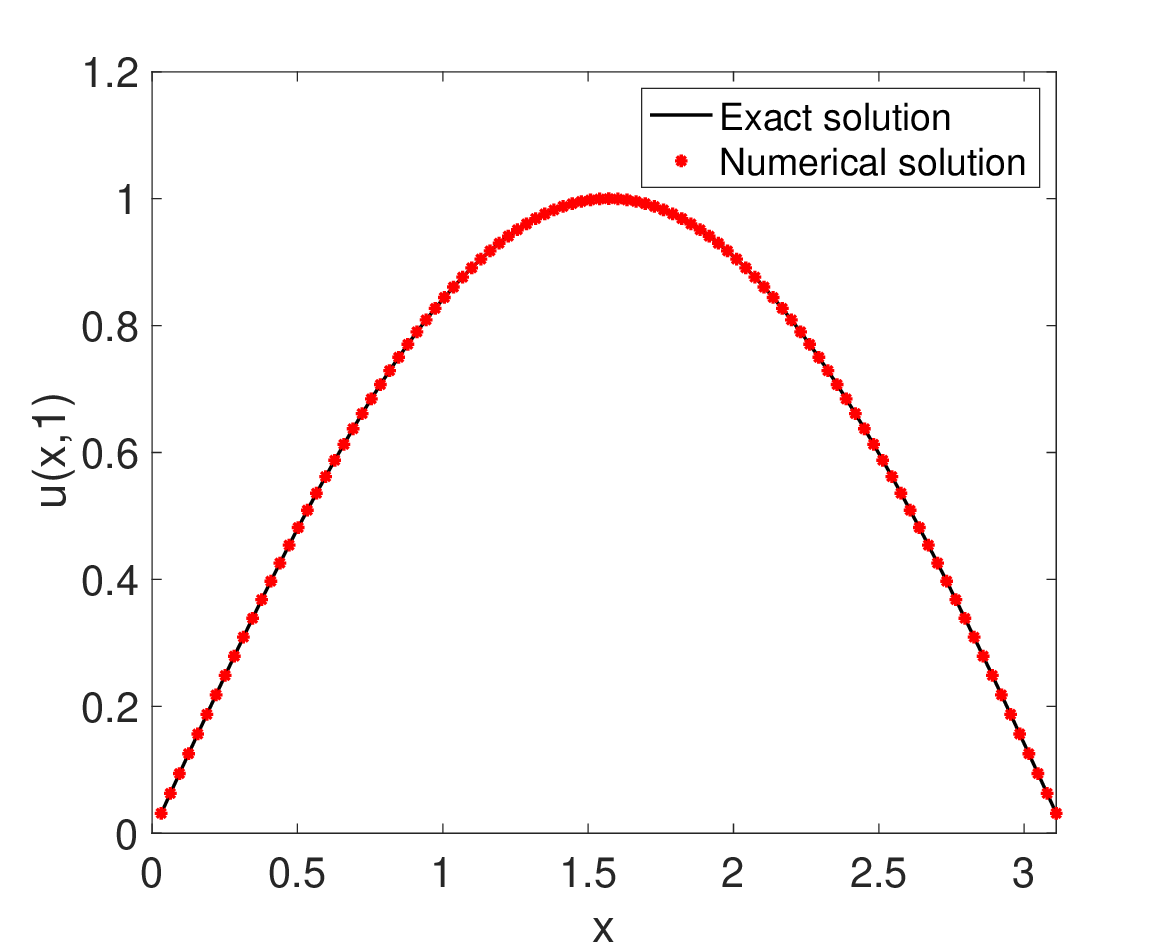}
			\label{fig1}} \quad
			\subfigure[]{%
			\includegraphics[scale=0.36]{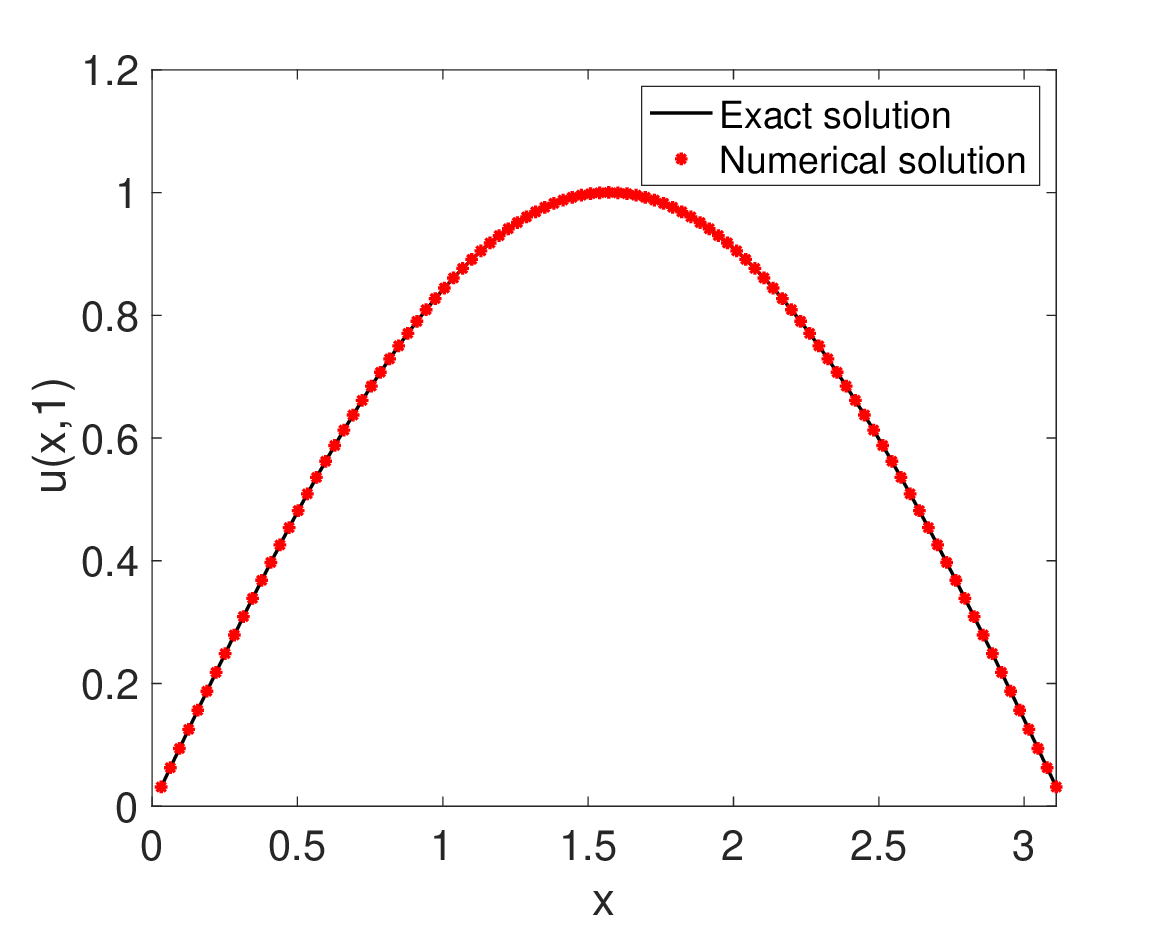}
			\label{fig2}} 
   
       \caption{Example \ref{ex1}. (a) Exact solution and numerical solution for $\alpha=0.3$ and $\varrho=1$; (b) Exact solution and numerical solution for $\alpha=0.3$ and $\varrho=30$.} 
   \label{figex1}
	\end{center}
\end{figure}

Table~\ref{table:3} displays that $\mathcal{O}(h^2)$ is the rate of convergence in spatial direction by fixing $K_t=1000$ and changing the numbers of splits of spatial interval.

\begin{table}[htbp]
    \centering
    \begin{tabular}{c| c c| c c| c c}
    \hline
   &\multicolumn{2}{|c|}{$\alpha=0.3$} &\multicolumn{2}{|c|}{$\alpha=0.6$} &\multicolumn{2}{|c}{$\alpha=0.8$}\\
\cline{2-7} 
   $K_x$   & $E_\infty$ & Order & $E_\infty$ & Order & $E_\infty$ & Order\\ \hline
   $10$ & $3.1000e-03$ &  & $2.1000e-03$ &  & $1.6000e-03$ & \\ 
   $20$ & $7.6951e-04$ & $2.0103$ & $5.3696e-04$ & $1.9675$ & $4.0646e-04$ & $1.9769$\\ 
   $40$ & $1.9239e-04$ & $1.9999$ & $1.3427e-04$ & $1.9997$ &$1.0165e-04$ & $1.9995$\\ 
   $80$ & $4.8097e-05$ & $2.0000$ & $3.3570e-05$ & $1.9999$ &$2.5416e-05$ & $1.9998$\\  
   $160$ & $1.2024e-05$ & $2.0000$ & $8.3930e-06$ & $1.9999$ &$6.3558e-06$ & $1.9996$\\    
    \end{tabular}
    \caption{Maximum spatial errors and order of convergence $(S_{rate})$ for $r=1$ and $\varrho=1$.}
    \label{table:3}
\end{table}

Next, we perform another numerical examination for TFDE problem \eqref{diff_eqn}-\eqref{ini_eqn} having a non-smooth behavior of the solution at initial time moment, i.e., $u(\cdot,t)\notin C^4[0,T]$. We observe the maximum nodal errors for different numbers of time intervals and confirm our theoretical convergence analysis. We also analyse the spatial convergence so that $\mathcal{O}(h^2+\tau^{\min\{4-\alpha,r\alpha\}})$ is proved numerically as the rate of convergence for the proposed scheme \eqref{h3}-\eqref{h5}.

\begin{example}\label{ex_2}
Let us consider the TFDE problem \eqref{diff_eqn} with following data:
\begin{equation*}
\begin{aligned}
u&(0,t) = u(1,t) = 0,~~~0 < t \leq 1,\\
&u(x,0) = 0,~~~0\leq x\leq1.
\end{aligned}
\end{equation*}
Here $f(x,t)=t^2\sin \pi x \big(\frac{\Gamma(3+\alpha)}{2}+\pi^2t^\alpha\big)$, and $u(x,t)= t^{2+\alpha}\sin \pi x$ is the exact solution. We have considered the value of  diffusion coefficient $\varrho$ to be 1 in this example. It is evident here that the exact solution has an initial-time singular behaviour, i.e., $u(\cdot,t)\notin C^4[0,T]$. 
\end{example}\noindent
We numerically solve the problem by keeping $K_x=10000$ fixed and taking different splits of temporal intervals on a uniform $(r=1)$ and graded mesh $(r=\frac{4-\alpha}{\alpha})$ using our newly developed scheme \eqref{h3}-\eqref{h5}. The maximum nodal errors \eqref{err_norm} and convergence orders \eqref{rate} using the proposed difference scheme are displayed in tables below at final time $t_{K_t} = T = 1$. Results in Table \ref{table:4} are calculated on the uniform mesh, which shows that the errors are minimal, but the convergence rate is not desired. It is clear from the results that Example \ref{ex_2} with a non-smooth solution on uniform mesh has error convergence of order $(\alpha)$, which is much less than the order of convergence $(4-\alpha)$ of the previous example having a smooth solution. \par
In Table \ref{table:5}, the numerical results agree precisely with the theoretical rate of convergence proved in Lemma \ref{trunclemma} using our proposed numerical scheme  \eqref{h3}-\eqref{h5}. We can see that the maximum errors are nominal and the optimal convergence order of $(4-\alpha)$ is achieved for the choice of grading constant $r=(4-\alpha)/\alpha$. In fact, for a choice of $r\geq (4-\alpha)/\alpha$, we can always get the optimal convergence order of $(4-\alpha)$.

\begin{table}[ht]
    \centering
    \begin{tabular}{c| c c| c c| c c}
     \hline
  &\multicolumn{2}{|c|}{$\alpha=0.3$} &\multicolumn{2}{|c|}{$\alpha=0.6$} &\multicolumn{2}{|c}{$\alpha=0.8$}\\
\cline{2-7} 
  $K_t$ & $E_\infty$ & Order & $E_\infty$ & Order & $E_\infty$ & Order\\ \hline
   $10$ & $2.6610e-06$ & & $2.5619e-06$ &  & $5.4254e-07$ &   \\  
   $20$ & $3.0269e-07$ & $3.1361$ & $3.0002e-07$ & $3.0932$ & $1.1337e-07$ & $2.2587$ \\ 
   $40$ & $3.9856e-08$ & $2.9250$ & $4.0843e-08$ & $2.8769$ & $2.6454e-08$ & $2.0995$ \\ 
   $80$ & $1.2020e-08$ & $1.7294$ & $1.0281e-08$ & $1.9901$  & $8.8073e-09$ & $1.5867$ \\  
   $160$ & $9.3683e-09$ & $0.3596$ & $6.6545e-09$ & $0.6276$  & $5.0582e-09$ & $0.8001$ \\      
    \end{tabular}
    \caption{Maximum temporal errors and order of convergence $(T_{rate})$ for $r=1$.}
    \label{table:4}
\end{table}

\begin{table}[ht]
    \centering
    \begin{tabular}{c| c c| c c| c c}
    \hline
      &\multicolumn{2}{|c|}{$\alpha=0.4$} &\multicolumn{2}{|c|}{$\alpha=0.6$} &\multicolumn{2}{|c}{$\alpha=0.8$}\\
\cline{2-7} 
   $K_t$   & $E_\infty$ & Order & $E_\infty$ & Order & $E_\infty$ & Order\\ \hline
   $10$ & $9.7738e-04$ &  & $7.2723e-04$ &  & $4.3321e-04$ & \\  
   $20$ & $1.4429e-04$ & $2.7599$ & $9.7544e-05$ & $2.8983$ & $5.8820e-05$ & $2.8807$\\ 
   $40$ & $1.5987e-05$ & $3.1740$ & $1.0963e-05$ & $3.1534$ &$7.1073e-06$ & $3.0489$\\ 
   $80$ & $1.5462e-06$ & $3.3701$ & $1.1312e-06$ & $3.2767$ &$8.0959e-07$ & $3.1340$\\  
   $160$ & $1.2864e-07$ & $3.5873$ & $1.0580e-07$ & $3.4184$ &$8.4873e-08$ & $3.2538$\\    
    \end{tabular}
    \caption{Maximum temporal errors and order of convergence $(T_{rate})$ for $r=\frac{4-\alpha}{\alpha}$.}
    \label{table:5}
\end{table}

The graphical illustration of comparison between the exact solution and numerical solution of Example \ref{ex_2} with diffusion coefficient $\varrho=1$ on uniform mesh $r=1$ and graded mesh $r=\frac{4-\alpha}{\alpha}$ is given in Figure \ref{fig3} and Figure \ref{fig4}, respectively.

\begin{figure}[htbp]\
	\begin{center}
		\centering
		\subfigure[]{%
			\includegraphics[scale=0.36]{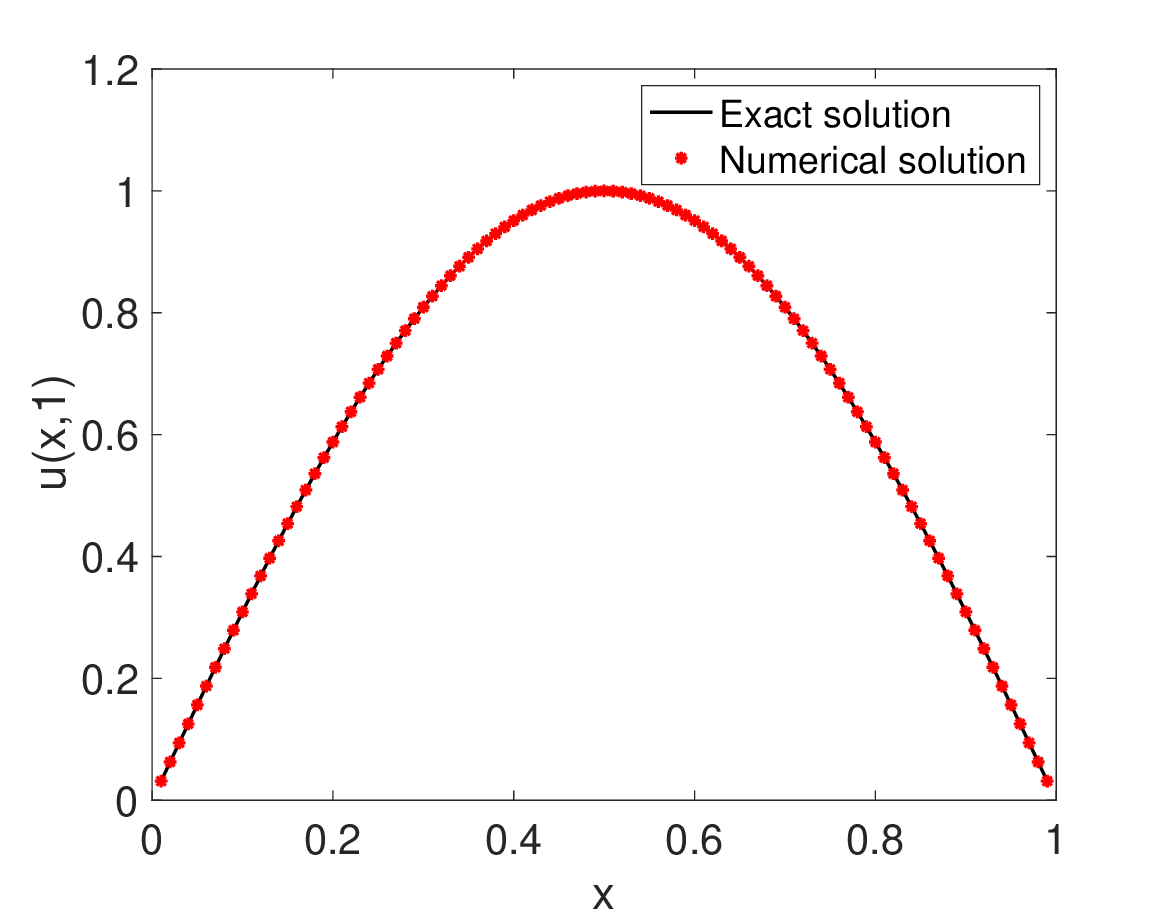}
			\label{fig3}} \quad
			\subfigure[]{%
			\includegraphics[scale=0.35]{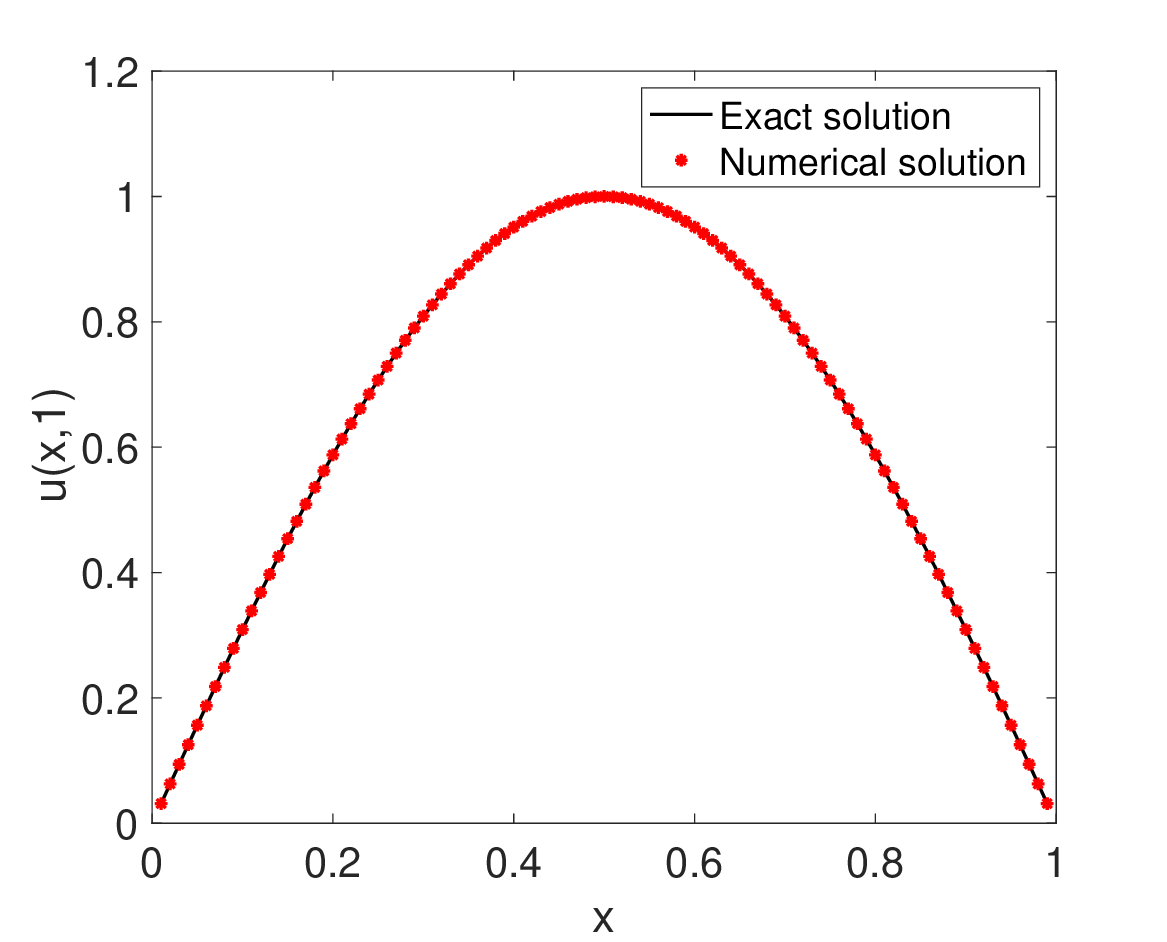}
			\label{fig4}} 
       \caption{Example \ref{ex_2}. (a) Exact solution and numerical solution for $\alpha=0.3$ and $r=1$; (b) Exact solution and numerical solution for $\alpha=0.3$ and $r=\frac{4-\alpha}{\alpha}$.} 
   \label{figEx1}
	\end{center}
\end{figure}

Table \ref{table:6} is the display of convergence in the spatial direction which shows that $\mathcal{O}(h^2)$ convergence rate is attained. This evaluation is done keeping $K_t=1000$ fixed and taking different splits of space interval.
  \begin{table}[ht]
    \centering
    \begin{tabular}{c| c c| c c| c c}
    \hline
  &\multicolumn{2}{|c|}{$\alpha=0.3$} &\multicolumn{2}{|c|}{$\alpha=0.6$} &\multicolumn{2}{|c}{$\alpha=0.8$}\\
\cline{2-7} 
   $K_x$   & $E_\infty$ & Order & $E_\infty$ & Order & $E_\infty$ & Order\\ \hline
   $10$ & $7.2000e-03$ &  & $6.9000e-03$ &  & $6.6000e-03$ & \\  
   $20$ & $1.8000e-03$ & $2.0000$ & $1.7000e-03$ & $2.0211$ & $1.6000e-03$ & $2.0444$\\ 
   $40$ & $4.4627e-04$ & $2.0120$ & $4.3120e-04$ & $1.9791$ &$4.1171e-04$ & $1.9584$\\ 
   $80$ & $1.1158e-04$ & $1.9998$ & $1.0778e-04$ & $2.0003$ &$1.0291e-04$ & $2.0002$\\  
   $160$ & $2.7924e-05$ & $1.9985$ & $2.6944e-05$ & $2.0001$ &$2.5727e-05$ & $2.0000$\\    
    \end{tabular}
    \caption{Maximum spatial errors and order of convergence $(S_{rate})$ for $r=1$.}
    \label{table:6}
\end{table}
\section{Conclusion}\label{conclusion}\noindent
In this paper, we derive a new high order approximation to the Caputo fractional derivative of order $\alpha$, $0<\alpha<1$, over non-uniform mesh. The algorithm applied here is based on the approximation of the integrand of the Caputo derivative by the Lagrange interpolating polynomials. This algorithm relies on the solution being continuously differentiable up to fourth order. However, the considered Caputo-type TFDE problem \eqref{diff_eqn}-\eqref{ini_eqn} typically has a solution that exhibits initial-time singularity, i.e. $u(\cdot,t)\notin C^4[0,T]$. Therefore the discretization on graded mesh produces a local truncation error of order $\min\{4-\alpha, r\alpha\}$, $r$ being the mesh grading parameter. The quantity $r=\frac{(4-\alpha)}{\alpha}$ is observed as the optimal grading parameter. Next, a high-order difference scheme for TFDE problem \eqref{diff_eqn}-\eqref{ini_eqn} is developed. The proposed scheme is proved to be unconditionally stable and convergent on the uniform mesh $(r=1)$. It is also found that the proposed scheme is uniquely solvable. At present, we are unable to provide proof for the stability of the scheme for $r>1$ due to the complexity of coefficients, but we are committed to making efforts in that direction. Numerical experiments are carried out to verify the theoretical results for smooth and non-smooth solutions. The results of the second example demonstrate clearly that the proposed scheme for a non-smooth solution on a uniform mesh converges with an order ($\alpha$), which is much lower than the order of convergence in the first example, i.e. ($4-\alpha$), where the solution was smooth. Also, our scheme for a non-smooth solution reflects better convergence on the graded mesh rather than the uniform mesh. This scheme outperforms all existing schemes for graded mesh in terms of accuracy and efficiency. The proposed scheme is developed for a general non-uniform mesh and can be applied and analyzed for any kind of grading over meshes.

\section*{Declarations}
\subsection*{\textbf{Funding and acknowledgements}}\noindent
The first author acknowledges the support provided by the Council of Scientific and Industrial Research, India, under grant number 09/086(1483)/2020-EMR-I. The fourth author acknowledges the support provided by the SERB, a statutory body of DST, India, under the award SERB--POWER  fellowship (grant  number SPF/2021/000103).
\subsection*{\textbf{Conflict of interest}}\noindent 
The authors declare that there are no financial and non-financial competing interests that are relevant to the content of this article.

\bibliographystyle{plain}
\bibliography{mybib}
\end{document}